\documentclass{amsart}[12 pt]

\usepackage{amsmath, amsfonts, amssymb, euscript, amscd, latexsym, mathrsfs, bm, color, bbm}
 \usepackage{titlesec}
 \usepackage[pdftex,colorlinks=true,
                       pdfstartview=FitV,
                       linkcolor=blue,
                       citecolor=blue,
                       urlcolor=blue,
           ]{hyperref}

\def\aa#1{ \begin{align*} #1 \end{align*} }
\def\aaa#1{ \begin{align} #1 \end{align} }
\def\mm#1{ \begin{multline*} #1 \end{multline*} }
\def\mmm#1{ \begin{multline} #1 \end{multline} }

\newtheorem{thm}{\sc Theorem}[section]
\newtheorem{lem}{\sc Lemma}[section]

\newcommand{\eps}{\varepsilon}
\newcommand{\pl}{\partial}

\newcommand{\lt}{\leqslant}

\newcommand{\sub}{\subset}

\newcommand{\al}{\alpha}

 \newcommand{\Gm}{\Gamma}
 \newcommand{\Dl}{\Delta}
 
 \newcommand{\La}{\Lambda}

\newcommand{\mc}{\mathcal}

\newcommand{\C}{{\rm C}}

\newcommand{\E}{\mathbb E}

\newcommand{\x}{\times}
\newcommand{\mto}{\mapsto}

\newcommand{\rf}{\eqref}
\newcommand{\bi}{\begin{itemize}}
\newcommand{\ei}{\end{itemize}}

\newcommand{\ovl}{\overline}

\DeclareMathOperator{\ind}{\mathbbm{1}}
\newcommand{\lap}{\Delta}
\newcommand{\nab}{\nabla}
\newcommand{\lb}{\label}
\newcommand{\fdot}{\,\cdot\,}

\def\Rnu{{\mathbb R}}

\def\Nnu{{\mathbb N}}

\def\Bl{{\mathbb B}}
\def\Tor{{\mathbb B}}

\def\ffi{\varphi}

\def\com#1{}

\def\begeq{\begin{equation} \begin{cases}} 
\def\endeq{ \end{cases} \end{equation}}
\def\eq#1{ \begeq #1 \endeq }
\def\bege{\begin{equation*} \begin{cases}} 
\def\ende{ \end{cases} \end{equation*}}

\long\def\symbolfootnote[#1]#2{\begingroup%
\def\thefootnote{\fnsymbol{footnote}}\footnote[#1]{#2}\endgroup}
\titleformat{\section}[hang]{\large\bfseries}{\thesection.}{1ex}{}{}
\titleformat{\subsection}[hang]{\normalsize\bfseries}{\thesubsection}{2ex}{}{}
\titleformat{\subsubsection}[hang]{\small\bfseries}{\thesubsubsection}{2ex}{}{}

\allowdisplaybreaks

\include{srctex.sty}

\usepackage{lipsum}
\usepackage{algorithmic}

\title[Burgers equation in the adhesion model]{Burgers equation in the adhesion model}

\author{Yuri Gliklikh}
\address{Voronezh State University, Universitetskaya pl. 1, Voronezh, Russia   }
\email{yeg@math.vsu.ru}

\author{Evelina Shamarova}
\address{Mathematics department, Federal University of Para\'iba, Jo\~ao Pessoa, Brazil}
\email{evelina@mat.ufpb.br}

\begin{document}

\maketitle

\vspace{-7mm}

\begin{abstract}
We prove the existence and uniqueness of a classical solution to
a  multidimensional non-potential  stochastic Burgers equation with H\"older continuous
initial data. Our motivation is the adhesion model in the theory of formation of the large-scale
structure of the universe. Importantly, we drop the assumption on the potentiality of the velocity flow that has been questioned
in physics literature.
\end{abstract}

\vspace{3mm}

{\footnotesize

\vspace{3mm}

}

\section{Introduction}
In this work, we are concerned with
a multidimensional  stochastic Burgers equation of the form
\aaa{
\label{burgers}
y(t,x) =  \ffi(x) + \int_0^t \big[\nu \Dl y(s,x) - (y,\pl_x)y(s,x)  \big] ds +  \eta(t,x)
 }
 with a view of its application to the adhesion model in cosmology. Here,
$(t,x) \in [0,T] \x \Rnu^d$, $\ffi$ is a H\"older continuous bounded initial function,
and $\eta(t,x)$ is a noise, rough in time and smooth in space.
The presence of the noise, however, is introduced for the sake of generality since 
in our main application
$\eta = 0$.

The adhesion model was introduced by Gurbatov and Saichev \cite{gurvatov-saichev1984} 
as a generalization of Zel'dovich's approximation \cite{zeldovich} aiming to represent the effect of gravitational 
sticking in the formation of the large-scale structure of the universe. 
Zel'dovich's approximation is based
on an inviscid potential Burgers equation, and is only valid in the early linear regime of the universe expansion.
It is essential that  in this part of the theory, the potentiality of the velocity field is used in the derivation of the basic 
equation of the model.
To formulate the adhesion model, the term $\nu \lap y$ was artificially
added to the (inviscid) Burgers equation, and then the model was tested numerically showing a qualitative and  a quantitative
agreement with a gravitational $N$-body  simulation \cite{coles} (a numerical solution of the equations of motion 
for $N$ gravitationally interacting particles).
It turned out that this agreement holds not only in the early linear  regime but also in the subsequent 
strongly non-linear regime,
 described as follows: the matter, previously concentrated in  Zel'dovich's pancakes, moves towards the faces, 
the edges, and the vertices
of the emerging mosaic structure, and then the latter is deformed due to the gravitational interaction \cite{gurbatov1}.
However, according to what was discussed in \cite{buchert92}, edges and vertices, representing 
high-density regions of the  structure (galaxy clusters and superclusters),
are associated with strong vortex flows. This means that the velocity flow is no longer potential.
Remark that the assumption on its potentiality
in the derivation of Zel'dovich's approximation is only valid in the early linear regime while in the adhesion model
the initial condition is still assumed potential.  This latter fact is our main motivation for studying
a non-potential Burgers equation which agrees with the adhesion model otherwise.  By the latter, we mean
that the initial function should be stationary against translations
and scale-invariant \cite{devilstaircase}. The above properties are satisfied by a 
fractional Brownian sheet $W^H(x)$.  Since we require the initial data to be bounded, we consider
$\ffi(x) = \zeta(x) W^H(x)$, where $\zeta(x)$ is a $\C^\infty$-cutting
function, i.e., a mollified indicator function, of a bounded domain
where the expansion of the universe takes place according to the adhesion model, and
whose diameter is compatible with the scale of validity of the latter \cite{gurbatov1}.




Thus, in this work, we are concerned with the existence and uniqueness of a classical solution to equation \rf{burgers} 
 when the initial function $\ffi(x)$ possesses just a H\"older regularity. According to the results of
\cite{meersch} (Theorem 6.2), the aforementioned property is satisfied by a fractional Brownian sheet.
Thus, the above-described choice of the initial data  is suitable for the adhesion model.

 A stochastic Burgers equation of form \rf{burgers} has been extensively studied 
in the literature over the past two decades  (see, e.g.,
 \cite{boritchev,  bec, bertini,  engle, frish, gotoh, gurbatov, gurbatov1, iturriaga}); however, in most cases this
 study was restricted to the one-dimensional or the potential case.  The interest to the potential case
 is mainly based on the existence of an exact solution by means of the Cole-Hopf 
 transformation.  In the non-potential case, it is only known   that a multidimensional stochastic
 Burgers equation possesses a solution in an $L_p$-space \cite{brzezniak}. However, an 
 $L_p$-solution is not a classical, and not even a continuous solution, 
 and, therefore, is not suitable for the adhesion model. Indeed,  the evolution of the large-scale structure of the universe
 is regarded as a continuous process of transport of the matter. 
 Thus, apart of our main important application, 
a classical solution to equation \rf{burgers} in the non-potential case represents an interesting mathematical question 
(see also the discussion in \cite{bec} on this topic).

Our pathway to obtaining a classical solution to \rf{burgers} is as follows. 
Equation \rf{burgers} is first reduced to a PDE with random coefficients. 
The sequence of classical solutions $y_m$ corresponding to a sequence $\ffi_m$ of mollified initial functions 
is shown to be bounded and H\"older continuous uniformly in $m$ which implies the existence of 
a converging subsequence. The limit function $y^{(1)}$ is then plugged into the non-linear term in \rf{burgers},
i.e., the non-linear term becomes $(y^{(1)}, \pl_x) y$, reducing \rf{burgers} to a linear equation.
For a linear PDE, in turn, it is well known (see, e.g., \cite{friedman}) that a H\"older regularity of the initial data
is sufficient for the existence of a classical solution.

 \section{Preliminaries}

\subsection{Existence of solution to a deterministic Burgers-type equation}

In this section,  we obtain the existence of solution to the Burgers-type equation
\aaa{
\lb{generalized-burgers}
\begin{cases}
\pl_t y(t,x) =  \nu \lap y(t,x) - (g(t,x,y),\pl_x)y(t,x) + f(t,x,y),\\
 y(0,x) = \ffi(x),
\end{cases}
} 
where the force $f$ is smooth in time and space variables and $\ffi$ is of 
class $\C^{2+\beta}_b(\Rnu^d)$, $\beta\in (0,1)$. By the latter, we understand the class of twice
differentiable functions possessing a H\"older continuous second derivative.
This result will be required in the next section.
 
Recall that the H\"older constants $[\Phi]^t_{\frac{\beta}2}$ and  $[\Phi]^x_{\beta}$  are defined as follows:
\aa{
[\Phi]^t_{\frac{\beta}2} = \sup_{\substack{t,t'\in [0,T],\\ t\ne t'}} \frac{|\Phi(t,x,y) - \Phi(t',x,y)|}{|t-t'|^\frac{\beta}{2}}; \; \;
[\Phi]^x_{\beta} =  \sup_{\substack{x,x'\in\Rnu^d,\\ 0<|x-x'|<1}}  \frac{|\Phi(t,x,y) - \Phi(t,x',y)|}{|x-x'|^\beta};
}
$[\Phi]^y_\beta$  is defined likewise. 

Theorem \ref{lem11} below  (the existence theorem for equation \rf{generalized-burgers}) and its proof
are included here for the sake of
completeness since in \cite{lady} both are given only for the case one equation.
\begin{thm}
\lb{lem11}
Let the functions $f$, $g$, and $\ffi$ satisfy assumptions (i) -- (iii) below
\bi
\item[(i)]  $\ffi:\, \Rnu^d\to \Rnu^d$
  belongs to $\C^{2+\beta}_b(\Rnu^d)$, $\beta\in (0,1)$;
\item[(ii)] the partial derivatives $\pl_t f$, $\pl_x f$, $\pl_y f$, $\pl_t g$, $\pl_x g$, $\pl_y g$ are continuous;
\item[(iii)] the following estimates hold  on $[0,T] \x \Rnu^d\x\Rnu^d$:
\aa{
&(f(t,x,y), y) < \mc C(1+|y|^2), \\
& |g(t,x,y)| + |\pl_x g(t,x,y)| + |\pl_y g(t,x,y)|  + [g]^t_{\frac{\beta}2} \lt \La_1(|y|),\\
& |f(t,x,y)| + |\pl_x f(t,x,y)| + |\pl_y f(t,x,y)|  + [f]^t_{\frac{\beta}2} \lt \La_2(|y|),
 }
where  $\mc C$ is a constant and 
 $\La_1(\fdot)$ and $\La_2(\fdot)$ are positive
non-decreasing functions. 
\ei
Then, there exists a $\C^{1,2}_b$-solution $y(t,x)$ to problem \rf{rforced}.
Further, the bound for $|y(t,x)|$ depends only on $\mc C$, $T$, and $\sup_{x\in \Rnu^d}|\ffi(x)|$.

\end{thm}
\begin{proof}
Consider the initial-boundary value problem
\aaa{
\lb{initial-boundary}
\begin{cases}
\pl_t y(t,x) = \nu \Dl y(t,x) - (g(t,x,y),\nab)y(t,x) + \zeta(x)f(t,x,y), \\
 y(0,x) = \ffi(x)\zeta(x), \quad y(t,x)\Big|_{\pl B_r} = 0,
 \end{cases}
} 
where $B_r$ is an open ball of radius $r>1$, $\pl B_r$ is its boundary,
and $\zeta(x)$ is a smooth cutting function of the ball $B_{r-1}$ with the properties
$\zeta(x) = 1$ if $x\in B_{r-1}$, $\zeta(x) = 0$ if $x\notin B_r$, $0\lt \zeta(x)\lt 1$,
$\zeta(x)$ is bounded together with its derivatives of all orders.
Consider, for example, a mollified indicator function of the ball $B_{r-\frac12}$: $\zeta = \ind_{B_{r-\frac12}} \ast \, \rho_{\frac12}$,
where $\rho_{\frac12}$ is the standard mollifier supported on the ball of radius $\frac12$.
The class of initial-boundary value problems for systems of quasilinear parabolic PDEs,
which includes problem \rf{initial-boundary}, was considered  by Ladyzhenskaya et al. in \cite{lady}
(Theorem 7.1, p. 596).  By Theorem 7.1 of \cite{lady},  problem \rf{initial-boundary}
possesses a unique solution on $[0,T]\x \ovl B_r$ which belongs to 
the H\"older space $\C^{1+\frac{\beta}2,2+\beta}([0,T]\x \ovl B_r)$ with the norm
\mm{
\|y\|_{\C^{1+\frac{\beta}2,2+\beta}([0,T]\x \ovl B_r)} =
\|y\|_{\C^{1,2}([0,T]\x \ovl B_r)} + \sup_{t\in [0,T]}[\pl_t y]_{\beta}^x + 
\sup_{t\in [0,T]}[\pl^2_{xx} y]_{\beta}^x \\
+ \sup_{x\in \ovl B_r}[\pl_t y]_{\frac{\beta}2}^t + \sup_{x\in \ovl B_r}[\pl_x y]_{\frac{1+\beta}2}^t
+ \sup_{x\in \ovl B_r}[\pl^2_{xx} y]_{\frac{\beta}2}^t.
} 
First of all, we note that
by Theorem 6.1 (p. 592) from \cite{lady}, there exists a bound for the gradient $\pl_x y(t,x)$ that only
depends on $\nu$, $\mc C$, $T$, and the bounds for $|\ffi(x)|$ and $|\nab \ffi(x)|$.
Also, according to the results of $\S$7 from \cite{lady}, the bound for $y(t,x)$
depends only on $\mc C$, $T$,  and the bound for  $|\ffi(x)|$.

Furthermore,  by Theorem 5.1 from \cite{lady} (Chapter VII), 
the $\C^{1+\frac{\beta}2,2+\beta}([0,T]\x \ovl B_{r-1})$-norm of the solution $y(t,x)$ possesses 
a bound that depends only on $\mc C$, $\nu$, $T$,
and on the H\"older norm $\|\ffi\|_{\C^{2+\beta}_b(\Rnu^d)}$. 
Note that the bound for this norm does not depend on the radius $r$ of $B_r$.

To prove the existence of solution to \rf{rforced}, we employ
the diagonalization argument similar to the one presented in \cite{lady} (p. 493) 
for the case of one equation.
Take a closed ball $\ovl B_R$ of radius $R$. Let $y_r(t,x)$ be the solution to
problem \rf{initial-boundary} in the ball $B_{r+1}$ with $r>R$. 
Since the H\"older norms 
$\|y_r\|_{\C^{1+\frac{\beta}2,2+\beta}([0,T]\x \ovl B_r)}$ possess a bound not depending on $r$, then, by Arzel\`a-Ascoli's theorem,
the family of functions $y_r(t,x)$, parametrized by $r$, is relatively compact in $\C^{1,2}([0,T]\x \ovl B_R)$. Therefore,
we can find a sequence $\{y_{r_n}\}$ which converges in $\C^{1,2}([0,T]\x \ovl B_R)$.
 We can also find a further subsequence $\{y^{(1)}_{r_n}\}$
that converges in $\C^{1,2}([0,T]\x \ovl B_{R+1})$. Proceeding this way,
we find a subsequence $\{y^{(k)}_{r_n}\}$ that converges in $\C^{2,1}([0,T]\x \ovl B_{R+k})$.
It remains to note that the diagonal subsequence  $\{y^{(n)}_{r_n}\}$ 
converges at each point of $[0,T] \x \Rnu^d$ to a function $y(t,x)$,
while its derivatives $\pl_t y^{(n)}_{r_n}$, $\pl_x y^{(n)}_{r_n}$, and $\pl^2_{xx} y^{(n)}_{r_n}$
converge to the corresponding derivatives of $y(t,x)$.
Clearly, $y(t,x)$ is a solution to \rf{rforced}. 
Since the  $\C^{1,2}$-norm of each function $y^{(n)}_{r_n}(t,x)$ has
the same bound, the solution $y(t,x)$ belongs to $\C^{1,2}_b([0,T]\x \Rnu^d)$.
 Moreover, the bound for $y(t,x)$ depends 
only on  $\mc C$, $T$, and the bound for $|\ffi(x)|$;
the bound for $\pl_x y(t,x)$ depends on $\nu$, $\mc C$, $T$, and the bounds for $|\ffi(x)|$
 and $|\nab \ffi(x)|$.

\end{proof}

\subsection{Examples of the noise process}

\textit{Example 1: Stochastic integral.} $\eta(t,x) = \int_0^t g(s,x) dB_s = \sum_{i=1}^l \int_0^t g_i(s,x)dB^i_s$, where 
$g_i(t,x) \in \C^{0,2+\beta}([0,T]\x\Rnu^d)$, $\beta \in (0,1)$, is such that $g_i(t,x) = g_i(t,x) \zeta(x)$ with
 $\zeta(x)$ being a $\C^\infty$-cutting function of the ball $\Tor_R = \{x\in \Rnu^d: |x| < R\}$ described
 in the previous subsection; further,
$B^i_t$ are independent real-valued Brownian motions, and the stochastic integral is defined for each $x\in\Rnu^d$.

\begin{lem}
\lb{lem99}
There is a version of 
the stochastic integral $\int_0^t g(s,x) dB_s$ which belongs to  $\C^{0,2}([0,T] \x \Rnu^d)$.
\end{lem}

The proof of Lemma \ref{lem99} follows from Kolmogorov's continuity theorem.
%
%

\textit{Example 2: $x$-regularized space-time white noise.}  Let $\dot W^i(t,x)$, $i=1,\ldots, d$, 
be independent space-time white noises, and let
$\dot W^i_\eps(t,x)$ be a regularization in $x$ of $\dot W^i(t,x)$, that is,
$\dot W^i_\eps(t,x) = (\dot W^i(t,\fdot) \ast \rho_{\eps})(x)$, where 
$\rho_{\eps}$ is a standard mollifier supported on the ball of radius $\eps$.
Alternatively, one can write 
$W^i_\eps(t,x) = (W^i(t,\fdot) \ast \pl^{n}_{x_1\ldots x_n}\rho_{\eps})(x)$, where $W^i(t,x)$  
is an $(n+1)$-parameter Brownian sheet on $\Rnu^d$.
Since we are interested in noises of class $\C^{0,2}_b(\Rnu^n)$, define $\dot \eta^i(t,x)$ as $\dot W^i_\eps(t,x) \zeta(x)$,
where $\zeta(x)$ is a $\C^\infty$-cutting function of a bounded domain $D\sub\Rnu^d$.

 \section{Classical solution to equation \rf{burgers}}
 
 In this section, we will prove the existence of a classical solution to equation \rf{burgers}
 under  assumptions (A1) and (A2):
 \bi
  \item[\bf (A1)] $\eta(t,x)$ is of class $\C^{0,2}_b([0,T]\x\Rnu^d)$;
 \item[\bf (A2)] the initial data $\ffi(x)$ is of class $\C^\beta_b(\Rnu^d)$, $\beta\in (0,1)$.
 \ei

\begin{thm}
\lb{thm3.1}
Let (A1) and (A2) be fulfilled. Then equation \rf{burgers} has a solution.
\end{thm}
\begin{proof}
\textit{Step 1. Transformation to a Burgers-type PDE.}
The substitution $\hat y(t,x) = y(t, x) - \eta(t,x)$
transforms \rf{burgers} to the following Burgers-type equation 
\aaa{
\lb{rforced}
\begin{cases}
\pl_t \hat y(t,x) =  \nu \lap \hat y(t,x) - (\hat y +\eta,\pl_x) \hat y(t,x) + f(t,x,\hat y),\\
\hat y(0,x) = \ffi(x),
\end{cases}
}
where  $f(t,x,y) = \nu\lap\eta(t,x) - \pl_x  \eta(t,x)\eta(t,x) - (y, \pl_x)\eta(t,x)$.

\textit{Step 2. Mollification.}
Let $\ffi_m$ be a mollification of $\ffi$,  
and $\eta_m(t,x)$ be a mollification of $\eta(t,x)$ in $(t,x)$, both by means of standard mollifiers.
Consider the problem
\aaa{
\lb{rforced1}
\begin{cases}
\pl_t \hat y(t,x) =  \nu \lap \hat y(t,x) - (\hat y +\eta_m,\pl_x) \hat y(t,x) + f_m(t,x,\hat y),\\
\hat y(0,x) = \ffi_m(x),
\end{cases}
}
where $f_m(t,x,y) = \nu\lap\eta_m(t,x) - \pl_x  \eta_m(t,x)\eta_m(t,x)- (y, \pl_x)\eta_m(t,x)$. 
By Theorem \ref{lem11}, the above problem
has a unique $\C^{1,2}_b([0,T]\x \Rnu^m)$-solution $y_m(t,x)$.

\textit{Step 3. Uniform boundedness of $y_m(t,x)$.} 
Note that $\bar y_m(t,x)=y_m(T-t,x)$
 is the unique $\C^{2,1}_b([0,T],\Rnu^d)$-solution to 
\aaa{
\lb{backward-pde2}
y(t,x) = \ffi_m(x) + \int_t^T \big[ \nu \lap  y(s,x) - ( y+\bar\eta_m, \pl_x) y(s,x)  + \bar f_m(s,x,y)] ds,
}
where $\bar f_m(t,x,y) = f_m(T-t,x,y)$, $\bar \eta_m(t,x) = \eta_m(T-t,x)$.

Consider the forward-backward SDE (FBSDE) associated to \rf{backward-pde2}:
\aaa{
\lb{fbsde-alt}
\begin{cases}
X^{\tau,x}_t = x+ \int_\tau^t (Y^{\tau,x}_s +\bar \eta_m(s,X^{\tau,x}_s) )ds + \sqrt{2\nu}\, (W_t-W_\tau), \\
Y^{\tau,x}_t = \ffi_m(X^{\tau,x}_T) + \int_t^T \bar f_m(s,X^{\tau,x}_s, Y^{\tau,x}_s) ds 
 - \sqrt{2\nu} \int_t^T Z^{\tau,x}_s dW_s.
\end{cases}
}
Let $X^{\tau,x}_t$ be 
a solution to
\aa{
X^{\tau,x}_t = x+ \int_\tau^t (\bar y_m(s,X^{\tau,x}_s) +\bar \eta_m(s,X^{\tau,x}_s)) ds + \sqrt{2\nu}\, (W_t-W_\tau).
}
It is known that (see, e.g, \cite{ma}) 
\aaa{
\lb{form2}
Y^{\tau,x,m}_t = \bar y_m(t, X^{\tau,x}_t), \qquad  Z^{\tau,x,m}_t = \pl_x \bar y_m(t,X^{\tau,x}_t)
}
is a solution to \rf{fbsde-alt}.
In what follows,
for any function $\phi(t,x)$, $\nab \phi(t,x)$ will denote
$\pl_x \phi(t,x)$  not to confuse with the partial derivative w.r.t. the upper index $x$ in $(X^{\tau,x}_t, Y^{\tau,x}_t, Z^{\tau,x}_t)$.
Also, for simplicity of notation, sometimes we skip the upper index $\tau,x$. 

It\^o's formula, applied to the BSDE in \rf{fbsde-alt}, implies
\mmm{
\E |Y^m_t|^2  \lt  \E|\ffi_m(X_T)|^2 + \int_t^T \E|\nu\lap\bar \eta_m(t,X_s) - \nab\bar \eta_m(t,X_s)\bar \eta_m(t,X_s)|^2 ds \\
+ \int_t^T\E (1+  2|\nab \bar\eta_m(s,X_s)|)|Y^m_s|^2 ds.
}
Since $\bar \eta_m$, $\nab \bar \eta_m$, $\lap \bar \eta_m$, and $\ffi_m$ are bounded uniformly in $m$,
by Gronwall's inequality, $\E |Y^{\tau,x,m}_t|^2$, and, consequently, $\bar y_m(\tau,x)$, is bounded over 
$[0,T)\x\Rnu^d$ uniformly in $m$. 

\textit{Step 4. $y_m(t,x)$ contains a pointwise converging subsequence.}
 By Theorem 1.1 (Chapter V) in \cite{lady},  under (A1) and (A2), $y_m(t,x)$ is of class $\C^{\frac{\al}2,\al} ([0,T]\x\Tor_R)$ for any ball 
 $\Tor_R =\{|x| < R\}$,
  where the exponent $\al$ and the H\"older constant does not depend on $m$.
  By the Arzel\'a-Ascoli theorem, the family $y_m(t,x)$ is relatively compact in $\C([0,T],\ovl\Tor_R)$. Therefore, there exists a subsequence $y_{m_k}(t,x)$
  and a function $y^{(1)}(t,x) \in \C([0,T],\ovl\Tor_R)$ which is the uniform limit of $y_{m_k}(t,x)$ on $[0,T]\x\ovl\Tor_R$.
  Choosing a further subsequence of $y_{m_k}(t,x)$, we prove that there exists its uniform limit (which we again denote by $y^{(1)}(t,x)$)
  on $[0,T]\x\ovl\Tor_{R+1}$.
  For every $N\in\Nnu$, in a finite number of steps, we find a subsequence of $y_m$ converging uniformly on $[0,T]\x\ovl\Tor_{R+N}$.
  Its limit function, again, will be denoted by $y^{(1)}(t,x)$. Passing to the diagonal subsequence, we obtain a subsequence 
  converging to $y^{(1)}(t,x)$ pointwise on $[0,T]\x\Rnu^d$ and uniformly on each closed ball $\ovl\Bl_R$ centered at the origin.
  Remark that since $y_m$ is of class $\C^{\frac{\al}2,\al} ([0,T]\x \Bl_R)$, 
  so is $y^{(1)}(t,x)$.
  Also remark that if $R$ is fixed, by aforementioned  Theorem 1.1 (Chapter V) in \cite{lady}, 
  due to the global boundedness of $\eta_m$, $\pl_x\eta_m$, and $f_m$  (uniform in $m$),
  the H\"older constant
  for  $y_m$  in a ball $K_R=\{x: |x-a| < R\}$ does not depend on the center $a$ of the ball. This implies that this  H\"older constant
  is uniform in $m$ over $\Rnu^d$. Therefore, $y^{(1)} \in \C^{\frac{\al}2,\al} ([0,T]\x\Rnu^d)$.

  The pointwise converging (diagonal) subsequence constructed above, we again denote by $\{y_m\}$, i.e., 
  $y^{(1)}(t,x) = \lim_{m\to \infty} y_m(t,x)$.
  
  \textit{Step 5. The limit of $y_m$ is a solution to equation \rf{rforced}.}
 Consider a linear PDE with respect to $v$:
\eq{
\lb{cauchy-v} 
\pl_t v(t,x) =  \nu \lap v(t,x) - (y^{(1)} +\eta, \pl_x)v(t,x) + f(t,x, y^{(1)}(t,x)),\\
v(0,x) = \ffi(x).
}
By Theorem 12 (Chapter 1, Paragraph $\S$7) in \cite{friedman}, 
\aaa{
\lb{y2}
y^{(2)}(t,x) = \int_{\Rnu^d} G(t,x;0,z) \ffi(z) dz  + \int_0^t \int_{\Rnu^d} G(t,x;s,z)  f(t,z, y^{(1)}(t,z)) dz ds
}
is a solution to problem \rf{cauchy-v} which is continuous on $[0,T]\x\Rnu^d$ and is of class $\C^{1,2}$ on 
 $(0,T]\x\Rnu^d$, where $G(t,x;s,z)$ is the fundamental solution for the differential operator 
 $\nu \lap  - (y^{(1)} +\eta, \pl_x)$.
 
We use the associated BSDE to show that $y^{(2)}$ is also  a limit of $y_m$. This will prove
 that $\hat y(t,x) = y^{(2)}(t,x) = y^{(1)}(t,x)$ is a solution to problem \rf{rforced}, and, therefore,
 $y(t,x) = \hat y(t,x) + \eta(t,x)$ is a solution to original equation \rf{burgers}.

Define $B^{\tau,x}_t = x+\sqrt{2\nu}\, (W_t-W_\tau)$,   $\bar y^{(1)} (t,x) = y^{(1)} (T-t,x)$, and
note (see, e.g., \cite{ma}) that
  \aa{
  Y^{\tau,x}_t = y^{(2)} (T-t,B^{\tau,x}_t), \qquad
 Z^{\tau,x}_t = \pl_x y^{(2)}(T-t,B^{\tau,x}_t)
 }
 satisfies the BSDE
\mmm{
\lb{fbsde4}
Y^{\tau,x}_t =  y^{(2)}(\eps, B^{\tau,x}_T) + \int_t^{T-\eps} \big[\bar f(s,B^{\tau,x}_s, \bar y^{(1)}(s,B^{\tau,x}_s)) \\
 - Z^{\tau,x}_s (\bar y^{(1)}(s,B^{\tau,x}_s)+\bar \eta(s,B^{\tau,x}_s)) 
\big]ds 
- \sqrt{2\nu} \int_t^{T-\eps} Z^{\tau,x}_s dW_s,
} 
where $\eps>0$ is a sufficiently small number that appears due to the fact that $y^{(2)}$ may not be differentiable in $t$ at zero.
Further, since $f$ and $\ffi$ in \rf{y2} are bounded,
then $y^{(2)}$ is  also bounded. Indeed, the fundamental solution $G(t,x;s,z)$ possesses 
estimates by Gaussian densities (see, e.g., \cite{friedman}, Chapter 1, formula (6.12)). This implies the boundedness of  $Y^{\tau,x}_t$,
and, therefore, by It\^o's formula, there exists a constant $k>0$, independent of $\eps$, such that
\aa{
\E \int_t^{T-\eps} |Z_s|^2 ds \lt k.
}
Note that the above identity  also  holds for $\eps = 0$ which allows us to consider \rf{fbsde4} for $\eps = 0$.
On the other hand, 
\aaa{
\lb{form4}
Y^{\tau,x,m}_t = \bar y_m(t, B^{\tau,x}_t),  \qquad  Z^{\tau,x,m}_t = \pl_x \bar y_m(t,B^{\tau,x}_t)
}
is a solution to 
 \mmm{
\lb{fbsde}
Y^{\tau,x,m}_t = \ffi_m(B^{\tau,x}_T) +
 \int_t^{T} \big[\bar f_m(s,B^{\tau,x}_s,  Y^{\tau,x,m}_s) \\  - Z^{\tau,x,m}_s (Y^{\tau,x,m}_s + \bar\eta_m(s,B^{\tau,x}_s))
\big]ds 
- \sqrt{2\nu} \int_t^{T}  Z^{\tau,x,m}_s dW_s,
}
where $\bar y_m$ is the solution to \rf{backward-pde2}.
Remark that \rf{form4} is a different process than \rf{form2}, introduced earlier. However, we use the same symbols to simplify notation.
Since $Y^{\tau,x,m}_t$, $\bar\eta_m(t,B^{\tau,x}_t)$, and $\bar y^{(1)}(t,B^{\tau,x}_t)$ are globally bounded by a constant that does not depend on $m$,
 by the standard argument involving It\^o's formula, from \rf{fbsde4} and \rf{fbsde} we obtain
 \mmm{
 \lb{est8}
 \E|Y^{m}_t - Y_t |^2 + 2\nu\, \E\int_t^{T} |Z^{m}_s- Z_s|^2 ds
 \lt  \E| \ffi(B^{\tau,x}_T)  - \ffi_m(B^{\tau,x}_T)|^2\\
 +\epsilon\, \E \int_t^{T}  |Z^{m}_s - Z_s|^2 ds + K \, \E \int_t^T |Y^{m}_s - Y_s|^2 ds \\
+ \E\int_t^T|Z_s|^2  \big(|\bar y^{(1)} - \bar y_m|^2+ |\bar\eta - \bar\eta_m|^2\big)(s,B^{\tau,x}_s)  ds \\
+ \E \int_t^T  \big(|\bar y^{(1)} - \bar y_m|^2+ |\nab \bar\eta_m-\nab\bar \eta|^2
+  |\bar\eta_m-\bar \eta|^2 + |\lap \eta_m - \lap \eta|\big) (s,B^{\tau,x}_s)ds,
  }
  where $\epsilon>0$ is a number smaller than $2\nu$ and $K>0$ is a constant.
   Above, we skipped the upper index ${\tau,x}$ in $Y^{m,\tau,x}_t$,
  $Z^{m,\tau,x}_t$, $Y^{\tau,x}_t$, and $Z^{\tau,x}_t$ to simplify notation.
Note that the first term and the terms in the last two lines on the right-hand side go to zero by Lebesgue's dominated convergence theorem. 
  Therefore, by Gronwall's inequality, $ \E|Y^{m}_t - Y_t |^2 \to 0$ as $m\to \infty$, and hence,
 $\lim_{m\to \infty} y_m=y^{(2)}$ pointwise on $[0,T]\x \Rnu^d$. This proves that
 $y^{(1)}(t,x) = y^{(2)}(t,x)$. 
 
 \textit{Step 6. Uniqueness.} 
 Suppose there are two solutions  $y_1, y_2\in\C^{0,2}_b((0,T]\x\Rnu^d)$ to problem \rf{burgers},
 and let $Y^1_t$ and $Y^2_t$ be the solutions to the associated BSDEs of form \rf{fbsde4} (with $\eps=0$):
 \aa{
Y^{i}_t =  \ffi(B^{\tau,x}_T) + \int_t^{T} \big[\bar f(s,B^{\tau,x}_s, Y^i_s) 
 - Z^{i}_s (Y^i_s+\bar \eta(s,B^{\tau,x}_s)) 
\big]ds 
- \sqrt{2\nu} \int_t^{T} Z^{i}_s dW_s,
} 
$i=1,2$. The same argument as in Step 3 implies the boundedness of $y_1$ and $y_2$ by the same constant. 
Since $\int_0^T |Z^i_s|^2 ds \lt k$ for some constant $k$ (which follows from It\^o's formula), then
an estimate similar to \rf{est8} implies that $\E|Y^1_t - Y^2_t|^2 = 0$ by Gronwall's inequality. 
Hence, $y_1 = y_2$. 
The theorem is proved.
\end{proof}

Finally, we prove that the solution to \rf{burgers} is adapted to the same filtration as $\eta(t,x)$.

\begin{thm}
\lb{adapted}
Let $\mc G_t$ be a filtration and let $\eta(t,x)$ be $\mc  G_t$-adapted for each $x\in\Rnu^d$, 
Then, the solution 
$y(t,x)$ to \rf{burgers} is  also $\mc  G_t$-adapted.
\end{thm}
\begin{proof}
First, let us prove that if for each $(x,y)$, $\eta(t,x)$ and $f(t,x,y)$ are  $\mc  G_t$-adapted, 
so is the solution to \rf{rforced}.
We can prove the theorem for a smooth initial data and a smooth noise since according to the proof
of Theorem \ref{thm3.1}, the solution to \rf{rforced} can be approximated pointwise by solutions to \rf{rforced1}.

By \cite{lady}, the central argument in the proof of existence for initial-boundary value problem 
\rf{initial-boundary} is the Leray-Schauder theorem. According to the latter, the map
$\Gm: v\mto y$ defined by 
 \aaa{
\lb{initial-boundary1}
\begin{cases}
\pl_t y(t,x) = \nu \Dl y(t,x) - (v(t,x) + \eta(t,x),\nab)v(t,x) + \zeta(x)f(t,x,v), \\
 y(0,x) = \ffi(x)\zeta(x), \quad y(t,x)\Big|_{\pl B_r} = 0,
 \end{cases}
} 
has a fixed point. Here,
 the cutting function $\zeta(x)$ is defined as in the proof of Theorem \ref{lem11}. Hence,
\mmm{
y(t,x) = \int_{B_r} G(x,z,t) \ffi(z) \zeta(z) dz  \\ +\int_0^t \int_{B_r} 
G(x,z,t-s) \big( \zeta(z)f(s,z,v(s,z)) - (v(s,z) + \eta(s,z),\pl_z)v(s,z)\big) dz ds,
}
where $G(x,z,t)$ is the Green function for the heat equation $\pl_t y(t,x) = \nu \Dl y(t,x)$ on the ball $B_r$
(with zero boundary condition). If $v(t,x)$ is a $\C^{0,1}([0,T]\x B_r)$
deterministic function, then $y(t,x) = \Gm(v)(t,x)$ is clearly $\mc G_t$-adapted. 
Therefore, the fixed point of $\Gm$, which is the solution to \rf{initial-boundary1}, is $\mc G_t$-adapted.
 The solution to Cauchy problem \rf{rforced} is then $\mc G_t$-adapted as a pointwise limit (by construction) of solutions to initial-boundary
value problems \rf{initial-boundary1}. 

Observe that since $\eta(t,x)$ is $\mc G_t$-adapted for each $x$, $f(t,x,y)$, defined via $\eta(t,x)$ in Step 1 of the proof
of Theorem \ref{thm3.1}, is also  $\mc G_t$-adapted. The connection between the solutions to problems \rf{rforced}
and  \rf{burgers} implies that the latter is $\mc G_t$-adapted.
\end{proof}

\section*{Acknowledgements}
The research of the first-named author was partially supported by RFBR Grant N 18-01-00048.

\bibliographystyle{siamplain}

\end{document}